\documentclass[10pt, reqno]{amsart}

\usepackage{amssymb,latexsym,  amsmath}
\usepackage{eucal}
\newcommand{\supp}{\mbox{supp}}

\newtheorem{lemma}{Lemma}
\newtheorem{theorem}{Theorem}

\begin{document}

\title[On zero-divisors in group rings of groups with torsion]
{On zero-divisors in group rings of  groups with torsion}

\author{S.V. Ivanov and Roman Mikhailov}

\address{ Department of Mathematics\\
University of Illinois  \\1409 West Green Street\\ Urbana\\   IL
61801\\ USA} \email{ivanov@math.uiuc.edu}

\address{ Steklov  Mathematical Institute\\
Gubkina 8, Moscow, 119991\\ Russia}  \email{romanvm@mi.ras.ru}

\thanks{The first named author is supported in part by NSF grant
DMS 09-01782.}
%\date{}

\begin{abstract} Nontrivial pairs of zero-divisors in group rings  are introduced and discussed. A problem on the existence of nontrivial  pairs of zero-divisors in group rings of free Burnside groups of odd exponent $n \gg 1$ is solved in the affirmative. Nontrivial  pairs of zero-divisors are also found in group rings of free products of groups with torsion.
\end{abstract}
\maketitle

\section{Introduction}

Let $G$ be a group and $\mathbb Z[G]$ denote the group ring of $G$ over the integers. If $h \in G$ is an element of finite order $q > 1$ and  $X, Y \in \mathbb Z[G]$, then we have the following equalities in $\mathbb Z[G]$
\begin{align*}
& X (1-h)\cdot (1+h+\dots+h^{q-1}) Y =0 \ , \\
& X (1+h+\dots+h^{q-1})\cdot (1-h)Y =0 \ .
\end{align*}
Hence, $X (1-h)$ and $(1+h+\dots+h^{q-1}) Y$, $X (1+h+\dots+h^{q-1})$ and $ (1-h)Y$ are left and right zero-divisors of  $\mathbb Z[G]$ (unless one of them is 0 itself) which we call { trivial}  pairs of zero-divisors associated with an element $h \in G$ of finite order $q>1$.
Equivalently, $A, B \in \mathbb Z[G]$, with $AB = 0$, $A, B \ne 0$, is a  {\em trivial }  pair of  zero-divisors  in  $\mathbb Z[G]$  if   there are  $X, Y \in \mathbb Z[G]$ and $ h \in G$ of finite order $q >1$  such that either $A = X(1-h)$ and $B = (1+h+\dots+h^{q-1})Y$ or
$A = X(1+h+\dots+h^{q-1})$ and $B = (1-h)Y$.

An element $A \in \mathbb Z[G]$ is called a nontrivial left  (right)  zero-divisor if
$A$ is a left  (right, resp.)  zero-divisor and for every $B \in \mathbb Z[G]$   such that $B \ne 0$, $AB =0$, the pair $A, B$ is not
a trivial  pair of  zero-divisors.

The notorious Kaplansky conjecture on zero-divisors
claims that, for any  torsion-free group $G$, its integral group
ring $\mathbb Z[G]$ (or, more generally, its  group algebra $\mathbb F[G]$ over a field  $\mathbb F$)   contains no zero-divisors.
In this note, we are concerned with a more modest problem on the existence of zero-divisors in group rings of infinite groups with torsion that would be structured essentially different from the above examples of trivial   pairs of zero-divisors.  We remark in passing that  every pair of  zero-divisors in $\mathbb Z[G]$ is trivial whenever $G$ is cyclic (or locally cyclic).

Note that if $G$ is a finite group,  then every nonzero element $X$ in the augmentation ideal of $ \mathbb Z[G]$ is a left (right) zero-divisor, because the linear operator
$L_X : \mathbb Q[G] \to \mathbb Q[G]$, given by multiplication   $Y \to X Y$  ($Y \to Y X$, resp.), has a nontrivial kernel as follows from $\dim L_X(\mathbb Q[G]) < \dim \mathbb Q[G]$. Hence, $2- g_1 - g_2$, where $g_1, g_2 \in G$, is a left (right) zero-divisor of $ \mathbb Z[G]$ unless $g_1= g_2=1$.
On the other hand, the element $2- g_1 - g_2  \in \mathbb Z[G]$ is not a trivial left (right) zero-divisor unless $g_1, g_2$  generate a cyclic subgroup of $G$.  Hence, for a finite group $G$, the group ring $\mathbb Z[G]$ of $G$ contains no nontrivial zero-divisors if and only if $G$ is cyclic.
More generally, if $G$ is a group with a noncyclic finite subgroup $H$, then the element $2- h_1 - h_2  \in \mathbb Z[G]$, where $h_1, h_2 \in H$,  is a  nontrivial zero-divisor of $\mathbb Z[G]$ unless $h_1, h_2$  generate a cyclic subgroup of $H$ (for details see the proof of Theorem~\ref{th2}).

However, if $G$ is an infinite torsion (or periodic) group all of whose finite subgroups are cyclic, then the existence of nontrivial pairs of zero-divisors in  $\mathbb Z[G]$ is not clear. For instance, let $B(m,n)$ be the free Burnside group of rank $m$ and exponent $n$, that is, $B(m,n)$ is the quotient $F_m/F^n_m$ of a free group $F_m$ of rank $m$. It is known \cite{OB, Iv03} that if  $m \ge 2$ and  $n\gg 1$ is odd, then every noncyclic subgroup of  $B(m,n)$ contains a subgroup isomorphic to the free Burnside group $B(\infty,n)$ of countably infinite rank, in particular, every finite subgroup of $B(m,n)$  is cyclic. Note this situation is dramatically different for even $n \gg 1$, see \cite{Iv94}.

In this regard and because of other properties of $B(m,n)$, analogous to properties of absolutely free groups, see \cite{OB}, the first author asked the following question \cite[Problem 11.36d]{KNT}: Suppose $m \ge 2$ and odd $n\gg 1$. Is it true that every pair of zero-divisors in  $\mathbb Z[B(m,n)]$ is trivial, i.e., if $A B =0$ in $\mathbb Z[B(m,n)]$, then $A=X C$,\ $B =DY$, where  $X, Y$, $C, D \in \mathbb Z[B(m,n)]$ such that $C D=0$ and the set $\supp(C)\cup \supp(D)$ is contained in a cyclic subgroup of $B(m,n)$?

In this paper we will give a negative answer to this question by constructing a nontrivial pair of zero-divisors in $ \mathbb Z[B(m,n)]$  as follows.

\begin{theorem}\label{th1}
 Let $B(m,n)$ be  the free Burnside group of rank $m \ge 2$ and odd exponent $n\gg 1$, and $a_1, a_2$ be free  generators of $B(m,n)$. Denote $c := a_1a_2a_1^{-1}a_2^{-1}$ and let
 \begin{align*}
A   :=  (1 +c + \dots +c^{n-1})(1-a_1a_2a_1^{-1}) \,  , \ \
B    :=  (1-a_1)(1 +a_2 + \dots +a_2^{n-1}) \ .
\end{align*}
Then $AB = 0$ in $\mathbb Z[B(m,n)]$ and $A$, $B$  is a nontrivial pair of  zero-divisors in $\mathbb Z[B(m,n)]$.
\end{theorem}

It seems of interest to look at other classes of groups with torsion all of whose finite subgroups are cyclic and ask a similar question on the existence of nontrivial pairs of zero-divisors in their group rings.
From this viewpoint,
we  consider  free products of cyclic groups, all of whose finite subgroups are cyclic by the Kurosh subgroup theorem \cite{Kur},  and show the existence of nontrivial pairs of zero-divisors in their group rings. More generally, we will prove the following.

\begin{theorem}\label{th2}
Let a group  $G$ contain a subgroup isomorphic either to a finite noncyclic group or to the free product $C_{q}* C_{r}$, where  $C_n$ denotes a cyclic group of order $n$ (perhaps, $n = \infty$), and   $1< \min(q, r) < \infty$. Then  the integer  group ring  $ \mathbb Z[G]$ of $G$ has a nontrivial pair  of  zero-divisors.
\end{theorem}

On the one hand, in view of Theorems~\ref{th1}--\ref{th2}, one might wonder if there exists a nonlocally cyclic group $G$ with torsion without  nontrivial  pairs of zero-divisors in $ \mathbb Z[G]$, in particular,
whether there is a free Burnside group  $B(m,n)$, where $m, n >1$, with this property.   Note that, for every even $n \ge 2$ and  $m \ge 2$,  the free Burnside group  $B(m,n)$ contains a dihedral subgroup, hence, by Theorem~\ref{th2}, $\mathbb Z[B(m,n)]$  does have  a nontrivial pair  of  zero-divisors.

On the other hand, our construction of nontrivial pairs of zero-divisors in
$ \mathbb Z[C_{q}* C_{r}]$, where $1<q < \infty$, $r \in \{ 2, \infty\}$,  and $C_{q}= \langle a \rangle_{q}$  is generated by $a$, produces nontrivial  pairs of  zero-divisors of the form $A B= 0$, where $A =(1-a) U$, $B = U^{-1}(\sum_{i=1}^q a^i)$ and $U$ is a unit of $ \mathbb Z[C_{q}* C_{r}]$. Thus, our  nontrivial pairs of zero-divisors in $ \mathbb Z[C_{q}* C_{r}]$ are still rather restrictive  and could be named {\em primitive}.

Generalizing the definition of a trivial pair of zero-divisors, we say that
$A , B \in   \mathbb Z[G]$, where $A, B \ne 0$, $AB =0$, is a  {\em primitive} pair of  zero-divisors in $\mathbb Z[G]$ if there exists a unit $U$ of $\mathbb Z[G]$ such that
$A = X U$, $B = U^{-1} Y$, and $X$, $Y$ is a trivial pair of  zero-divisors in $ \mathbb Z[G]$.
One might conjecture that all pairs of  zero-divisors in $\mathbb Z[G]$ are primitive whenever $G$ is a free product of cyclic groups. Results and techniques of Cohn \cite{C1, C2}, see also  \cite{DS, Ger},  on units and zero-divisors in free products of rings could be helpful in investigation of   this conjecture.

\section{Three Lemmas}

\begin{lemma}\label{lem1}
Suppose that $G$ is a group, $ h \in G$, $H = \langle h \rangle$,
$X \in \mathbb Z[G]$, and $C \in \mathbb Z[H]$ is not invertible
in $\mathbb Z[G]$. Then, for every $g \in G$, the left coset $g H$
of $G$ by $H$ is either disjoint from $\supp (XC)$ or $| gH \cap
\supp (XC) | \ge 2$.
\end{lemma}

\begin{proof}   Denote $X = \sum_{i=1}^k x_i C_i$, where
$C_i  \in \mathbb Z[H]$ and $x_i \in G$, so that $ x_i H \ne x_j
H$ for $i \ne j$. Then $XC = \sum_{i=1}^k x_i C_i C$ and $\supp
(XC) = \bigcup_{i=1}^k x_i\supp (C_iC)$ is a disjoint union. Since
$C$ is not
 invertible in $\mathbb Z[H]$, $| \supp (C_iC)  | > 1$ and the result follows.
 \end{proof}
\medskip

Recall that a subgroup $K$ of a group $G$ is called {\em antinormal} if, for every $g \in G$,  the inequality
$g K g^{-1} \cap K \ne \{ 1\}$ implies that $g \in K$.

\begin{lemma}\label{lem2}
Suppose that $G$ is a group, $a, b \in G$, the elements $b, c: = aba^{-1}b^{-1}$ have order $n > 1$,  $d := aba^{-1}$, the cyclic subgroups $\langle c \rangle$, $ \langle a b^i  \rangle$,
$i =0,1\dots, n-1$, are nontrivial, antinormal, and $d \not\in  \langle c \rangle$, $c^j d \not\in  \langle a b^i \rangle$ for all $i, j \in \{ 0, 1\dots, n-1 \}$.  Then  equalities
\begin{align}\label{XC}
(1 +c + \dots +c^{n-1})(1-d) & = XC \ , \\ \label{DY}
(1-a)(1 +b + \dots +b^{n-1}) & = DY \  ,
\end{align}
where  $X, Y \in \mathbb Z[G]$,   $C, D \in \mathbb Z[H]$, $H$ is
a cyclic subgroup of $G$, and $CD =0$, are impossible.
\end{lemma}

\begin{proof}  Arguing on the contrary, assume that  equalities \eqref{XC}--\eqref{DY} hold true. Denote $H =  \langle h \rangle$.
Note that neither $C$ nor $D$ is invertible  in
 $\mathbb Z[H]$, because, otherwise, $CD =0$ would imply  that one of $C, D$ is  $0$ which contradicts  one of \eqref{XC}--\eqref{DY} and the assumptions $a \not\in \langle c \rangle$,   $c \neq 1$.

Hence, Lemma~\ref{lem1} applies to the equality \eqref{XC} and yields that the set
$$
\supp (XC) = \{ 1, c, \dots, c^{n-1},  d, cd, \dots, c^{n-1} d\}
$$
can be partitioned into subsets of cardinality  $> 1$  which are contained in distinct left cosets $g H$, $g \in G$.

Assume that $c^{i_1},  c^{i_2} \in gH$, where  $0 \le i_1 < i_2 \le n-1$. Then $c^{i_1 - i_2} = h^k \neq 1$ and, by antinormality of  $\langle c \rangle$, we have $h = c^i$ for some $i$. Since $d \in \supp (XC)$, it follows from Lemma~\ref{lem1} that $d h^j = d c^{ij} \in \supp (XC) $ with $h^j \neq 1$. Hence,
either $ d c^{ij} = c^{i'}$ or  $ d c^{ij} = c^{i'}a$ with $c^{ij} \ne 1$. In either case, we have a contradiction to  $d \not\in \langle c \rangle$ and  antinormality of $\langle c \rangle$.

Now assume that $c^{i_1}d,  c^{i_2}d \in gH$, where  $0 \le i_1 < i_2 \le n-1$. Then $d^{-1}c^{i_1 - i_2}d = h^k \neq 1$ and, by antinormality of  $\langle c \rangle$, we have $h = d^{-1}c^id$ for some $i$. Since $1 \in \supp (XC)$, it follows from Lemma~\ref{lem1} that $h^j =   d^{-1} c^{ij}d \in \supp (XC) $ with $h^j \ne 1$. Hence,
either $  d^{-1}c^{ij}d = c^{i'} \ne 1$ or  $ d^{-1}c^{ij}  d= c^{i'}d$. In either case, we have a contradiction to antinormality of $\langle c \rangle$ and $d \not\in \langle c \rangle$.

The contradictions obtained above prove that the foregoing partition of the set $\supp (XC) $ consists
of two element subsets so that one element belongs to $ \{ 1, c, \dots, c^{n-1} \}$ and the other belongs to $\{ d, cd, \dots, c^{n-1}d \}$. In particular, it follows from $1 \in  \{ 1, c, \dots, c^{n-1} \}$ that
\begin{equation}\label{h1}
h^{k_1} = c^{i_1 } d \neq 1
\end{equation}
for some $k_1, i_1$.

Applying a ``right hand" version of Lemma~\ref{lem1}, to the equality \eqref{DY}, we analogously
obtain that the set
$$
\supp (DY) = \{ 1, b, \dots, b^{n-1},  a, ab, \dots, ab^{n-1} \}
$$
can be partitioned into subsets of cardinality  $> 1$   which are contained in distinct right cosets $H g$, $g \in G$.

Assume that $b^{i_1},  b^{i_2} \in Hg$, where  $0 \le i_1 < i_2 \le n-1$. Then $b^{i_1 - i_2} = h^k \neq 1$ and, by antinormality of  $\langle b \rangle$, we have $h = b^i$ for some $i$. Since $a \in \supp (DY)$, it follows from the analog of Lemma~\ref{lem1} (in the ``right hand" version) that $ h^j a= b^{ij}a \in \supp (DY) $ with $h^j \ne 1$. Hence,
either $  b^{ij} a  = b^{i'}$ or  $ b^{ij}a = a b^{i'}$. In either case, we have a contradiction to  $c \ne 1$ and antinormality of $\langle b \rangle$.

Now assume that $ab^{i_1},  ab^{i_2} \in Hg$, where  $0 \le i_1 < i_2 \le n-1$. Then $ab^{i_1 - i_2}a^{-1} = h^k \neq 1$ and, by antinormality of  $\langle b \rangle$, we have $h = ab^i a^{-1}$ for some $i$. Since $1 \in \supp (DY)$, it follows from the analog of Lemma~\ref{lem1} that $h^j =  a  b^{ij} a^{-1}\in \supp (DY) $ with $h^j \ne 1$. Hence,
either $  ab^{ij}a^{-1}  = b^{i'} \ne 1$ or  $ ab^{ij}a^{-1} =  ab^{i'}$. In either case, we have a contradiction to antinormality of $\langle b \rangle$ and $ c \ne 1$.

The contradictions obtained above prove that the foregoing partition of the set $\supp (DY) $ consists
of two element subsets so that one element belongs to $ \{ 1, b, \dots, b^{n-1} \}$ and the other belongs to $\{ a, ab, \dots, ab^{n-1} \}$. In particular, it follows from $1 \in \{ 1, b, \dots, b^{n-1} \}$ that
\begin{equation}\label{h2}
h^{k_2} = ab^{i_2 } \neq 1
\end{equation}
for some $k_2, i_2$.

In view of equalities \eqref{h1}--\eqref{h2}, we obtain  $ c^{i_1 } d    ab^{i_2 }  = ab^{i_2 }  c^{i_1 } d  $.  Since the subgroup $\langle ab^{i_2 }  \rangle  $ is antinormal, we conclude that  $ c^{i_1 } d  \in \langle ab^{i_2 }  \rangle  $. This, however, is impossible by assumption and Lemma~\ref{lem2} is proved.
\end{proof}

\begin{lemma}\label{lem3}
Suppose  $a, b \in G$ are elements of a group $G$
such that the subgroup $\langle a, b \rangle$, generated by $a, b$, is isomorphic to  the free product $\langle a \rangle_q * \langle b \rangle_r$,
where $\langle c \rangle_s$ denotes a cyclic group of order $s$ generated by $c$ (perhaps, $s = \infty$), $1 < q < \infty$, $r \in \{ 2, \infty \}$, and $(q, r) = (2, 2)$ if $r =2$.
Then the elements
 \begin{gather}\label{AB} \textstyle{
 A := (1-a)\left( 1+ (1-a) b \left( \sum_{i=1}^q a^i  \right) \right) \ , } \\  \label{AB0} \textstyle{
 B := \left( 1-  (1-a) b \left( \sum_{i=1}^q a^i  \right) \right) \left( \sum_{i=1}^q a^i  \right)
 }
\end{gather}
 satisfy  $AB = 0$ and form a nontrivial pair of zero-divisors in $\mathbb Z[G]$.
\end{lemma}

\begin{proof}  Since
\begin{equation*}\textstyle{
\left( 1+ (1-a) b \left( \sum_{i=1}^q a^i  \right) \right) \cdot
\left( 1-  (1-a) b \left( \sum_{i=1}^q a^i  \right) \right) = 1 \ ,
}
\end{equation*}
it follows that $AB = (1-a) (\sum_{i=1}^q a^i ) = 0$  and $A, B \ne 0$, hence $A, B$ is a pair of  zero-divisors in $\mathbb Z[G]$.
We need to show that $A, B$ is a nontrivial pair of zero-divisors. Arguing on the contrary, assume that
$A, B$ is a trivial pair of zero-divisors in $\mathbb Z[G]$. Then there is an element  $h \in G$ of finite order $s >1$ and $X, Y \in \mathbb Z[G]$ such that either
\begin{equation}\label{AB1}\textstyle{
 A = X(1-h)   \quad   \mbox{and}  \quad  B = \left( \sum_{i=1}^s h^i  \right) Y }
\end{equation}
or
\begin{equation}\label{AB2}\textstyle{
 A = X\left( \sum_{i=1}^s h^i  \right)   \quad   \mbox{and}  \quad  B =  (1-h) Y  \ . }
\end{equation}

Let $\sigma : \mathbb Z[G] \to  \mathbb Z$ denote the augmentation homomorphism and $H = \langle h \rangle_s$.
It follows from definitions \eqref{AB}--\eqref{AB0} that $\sigma(A) = 0$ and  $\sigma(B) = q$. On the other hand, it follows that if \eqref{AB1} are true then $\sigma(A) = 0$ and if \eqref{AB2} hold then $\sigma(B) = 0$.  Hence, equalities \eqref{AB1} are true.
Looking again at   \eqref{AB}--\eqref{AB0}, we see that
\begin{equation}\label{suppA}
 \supp A  = \{   1, a, a^i b a^j \mid i \in \{ 0, 1\}, j \in \{ 0,1, \dots, q-1 \} \} \ .
\end{equation}
By Lemma~\ref{lem1}, $\supp A$ can be partitioned into subsets of cardinality  $> 1$   which are contained in distinct left cosets $g H$, $g \in G$.
Since $1 \in \supp A$, there is also an element $h^\ell \ne 1$ in $\supp A$.
Now we consider two cases: $r = \infty$ and $(q, r) = (2,2)$.

Suppose $r = \infty$. Since for all $i, j$ elements $a^i b a^j \in \supp A$
have infinite orders, it follows that $a = h^\ell$ for some $\ell$.

Assume   $(q, r) = (2,2)$. Then  \eqref{suppA} turns into
$$
 \supp A  = \{   1, a, b, aba,   ba,   ab  \} \ .
$$
Recall that $\supp A$ can be partitioned into some $k$ subsets $S_1, \dots, S_k$  of cardinality  $> 1$   which are contained in distinct left cosets $ g H$, $g \in G$. Hence, $k \le 3$. Note if $g_1, g_2 \in \{ 1, ba, ab \}$ are distinct, then $g_1^{-1}g_2$ has infinite order in  the free product $\langle a \rangle_2 * \langle b \rangle_2$, whence $g_1^{-1}g_2 \not\in H$ and
$g_1, g_2$ belong to different sets $S_1, \dots, S_k$. Therefore, $k =3$.

 Now we can verify that there is only one partition $\supp A = S_1\cup S_2\cup S_3$ such that $S_1 = \{ 1, g_2 \}$, $S_2 = \{ g_3, g_4 \}$, $S_3 = \{ g_5, g_6 \}$, $1 \in S_1$, $ba \in S_2$, $ab \in S_3$, and elements
 $g_2$, $g_3^{-1}g_4$, $g_5^{-1}g_6$ commute pairwise. This unique partition
 is the following:  $S_1 = \{ 1, a \}$, $S_2 = \{ b, ba \}$, $S_3 = \{ ab, aba \}$. Hence, $a = h^\ell$.

Thus in either case we have proved that  $a = h^\ell$  for some $\ell$. Then $\ell q = s$ and
 \begin{equation}\label{hah}  \textstyle{
\sum_{i=1}^s h^i  =  \left( \sum_{j=0}^{\ell- 1} h^j  \right)  \left( \sum_{k=0}^{q- 1} h^{\ell k}  \right) =  \left( \sum_{j=0}^{\ell- 1} h^j  \right)  \left( \sum_{k=1}^{q} a^{k}  \right) \ . }
\end{equation}
Hence,
$$ \textstyle{
A \left( \sum_{i=1}^s h^i \right)  = X (1-h) \left( \sum_{i=1}^{s} h^i  \right) = 0 \ . }
$$
On the other hand, it follows from  \eqref{hah} that
 \begin{align*}
 A \left( \sum_{i=1}^s h^i \right) & = A \left( \sum_{k=1}^{q} a^{k}  \right) \left( \sum_{j=0}^{\ell- 1} h^j  \right) = (1-a) b \left( \sum_{k=1}^{q} a^{k}  \right)^2 \left( \sum_{j=0}^{\ell- 1} h^j  \right) \\
& = q(1-a) b \left( \sum_{k=1}^{q} a^{k}  \right) \left( \sum_{j=0}^{\ell- 1} h^j  \right) = q(1-a) b   \left( \sum_{i=1}^s h^i \right)  \ .
\end{align*}

Hence,  $(1-a) b   \left( \sum_{i=1}^s h^i \right) = 0$ in $\mathbb Z[G]$ and,
for every product $b h^i$, $i =1, \dots, s$, there is $j$ such that $b h^i = a b h^j$. This equality implies that $b^{-1} a b = h^{i-j}$, hence $a = h^\ell$ commutes with  $b^{-1} a b$ in the free product $\langle a \rangle_q * \langle b \rangle_r$.  This is a contradiction which completes the proof.
 \end{proof}

\section{Proofs of Theorems}

{\em Proof of Theorem~1.}  Let $F_m = \langle b_1, b_2, \dots, b_m \rangle$
be a free group of rank $m$ with free generators $b_1, b_2, \dots, b_m$ and $B(m,n) = F_m/ F_m^n$ be the free $m$-generator Burnside group $B(m,n)$ of exponent $n$, where  $F_m^n $ is the (normal) subgroup generated by all $n$th powers of elements of $F_m$.  Let $a_1, a_2, \dots, a_m$ be  free generators of $B(m,n)$, where  $a_i$ is the image of $b_i$, $i = 1, \dots, m$, under the natural homomorphism $F_m \to B(m,n) = F_m/ F_m^n$.

Note that if $G = \langle g_1, g_2 \rangle$ is generated by elements $g_1, g_2$ and $G$ has exponent $n$, i.e. $G^n = \{ 1\}$, then $G$ is a homomorphic image of $B(m,n)$ if $m \ge 2$. Also, there is a nilpotent
group $G_{2,n} = \langle g_1, g_2 \rangle$ of exponent $n$ and class 2 in which elements $[g_1, g_2 ] := g_1g_2 g_1^{-1} g_2^{-1}$, $g_2$, $g_1g_2^i$,  $i = 0, \dots, n-1$, have order $n$. Therefore, elements $[a_1, a_2 ] := a_1a_2 a_1^{-1} a_2^{-1}$, $a_2$, $a_1a_2^i$,  $i = 0, \dots, n-1$, have order $n$ in $B(m,n)$ if $m \ge 2$. In addition, since $g_1g_2 g_1^{-1} \not\in  \langle  [g_1, g_2 ]  \rangle$ and $[g_1, g_2 ]^j  g_1g_2 g_1^{-1} \not\in  \langle  g_1 g_2^i  \rangle$ in $G_{2,n}$ for all $i,j \in \{ 0, \dots, n-1 \}$, it follows that
$a_1 a_2 a_1^{-1} \not\in  \langle  [a_1, a_2 ]  \rangle$ and $[a_1, a_2 ]^j  a_1a_2 a_1^{-1} \not\in  \langle  a_1 a_2^i  \rangle$  in  $B(m,n)$  for all $i,j \in \{ 0, \dots, n-1 \}$.

Recall that if $n\gg 1$ is odd (e.g. $n > 10^{10}$ as in \cite{Ol82}), then every  maximal cyclic subgroup of $B(m,n)$ is antinormal in
$B(m,n)$  (this is actually shown in the proof of \cite[Theorem 19.4]{OB},  similar arguments can be found in  \cite{Iv92, Iv05}).  Since cyclic subgroups  $\langle  [a_1, a_2 ]  \rangle$, $ \langle  a_1 a_2^i  \rangle$, $i = 0, \dots, n-1$, are of order $n$ and $B(m,n)$  has exponent $n$, it follows that these subgroups  $\langle  [a_1, a_2 ]  \rangle$, $ \langle  a_1 a_2^i  \rangle$, $i = 0, \dots, n-1$, are maximal cyclic and hence are antinormal. Now we can see that all conditions of Lemma~\ref{lem2}  are satisfied for elements $a= a_1$, $b= a_2$, $c = [ a_1, a_2]$, $d = a_1 a_2 a_1^{-1}$ of  $B(m,n)$. Hence, Lemma~\ref{lem2}  applies and yields that equalities \eqref{XC}--\eqref{DY} are impossible. Furthermore, it is easy to see that $(1+c + \dots +c^{n-1})(1-d) \ne 0$ because $c^i d \ne 1$, $i = 0, \dots, n-1$, and $(1-a)(1+b + \dots + b^{n-1}) \ne 0$ because $a b^j  \ne 1$, $j = 0, \dots, n-1$.

Finally, we need to show that
$$(1+c + \dots + c^{n-1})(1-d)(1-a)(1+b + \dots + b^{n-1}) = 0 \ . $$
Note $d = c b$ and $da = ab$, hence,  assuming that $ i_1, j_1, \dots, i_4, j_4$ are arbitrary integers that satisfy
$0\le  i_1, j_1, \dots, i_4, j_4 \le n-1$, we have
\begin{align*}
(1+& c   +\dots +  c^{n-1})(1-d)(1-a)(1+b + \dots + b^{n-1}
) = \\
 & = \left( \sum_{i_1} c^{i_1}  - \sum_{ i_2 } c^{i_2}d
 - \sum_{i_3 } c^{i_3}a  +\sum_{ i_4} c^{i_4}da
 \right) \left(  \sum_{ j_1 } b^{j_1} \right) =\\
  & = \left( \sum_{i_1} c^{i_1}  - \sum_{i_2} c^{i_2}cb
 - \sum_{i_3} c^{i_3}a  + \sum_{i_4} c^{i_4}ab
 \right) \left( \sum_{j_1} b^{j_1} \right) = \\
 & = \sum_{ i_1,j_1} c^{i_1}b^{j_1}  - \sum_{i_2,j_2 } c^{i_2+1}b^{j_2+1}
 - \sum_{i_3,j_3 } c^{i_3}a b^{i_3}   + \sum_{i_4,j_4 } c^{i_4}a b^{j_4+1} = 0  \ .
\end{align*}
Thus  $(1+c + \dots + c^{n-1})(1-d)$ and $(1-a)(1+b + \dots + b^{n-1})$ is a pair of  zero-divisors in  $\mathbb Z[B(m,n)]$ which is not trivial by Lemma~\ref{lem2} and Theorem~\ref{th1} is proved. \qed

\bigskip

The idea of the above construction of a nontrivial pair of  zero-divisors in  $\mathbb Z[B(m,n)]$ could be associated with Fox derivatives (which is somewhat analogous  to \cite{Iv99}, however, no mention of  Fox derivatives is made in \cite{Iv99})   and may be described as follows.  As above, let $F_2=F(b_1, b_2)$ be a free group with free generators
$b_1, b_2$. For $w \in F_2$, consider  Fox derivatives
  $\frac{\partial w}{\partial b_i}\in \mathbb Z[F_2],\ i=1,2$.
Then
\begin{align}\label{FD} \textstyle{
 w-1=\frac{\partial w}{\partial b_1}(b_1-1)+\frac{\partial w}{\partial b_2}(b_2-1)
 }
\end{align}
 in $\mathbb Z[F_2]$. Letting $w := [b_1, b_2]^n$, we observe that
 $\frac{\partial [b_1, b_2]^n}{\partial b_i} = (\sum_{j=0}^{n-1}[b_1, b_2]^j)\frac{\partial [b_1, b_2]}{\partial b_i}$, $i=1,2$. Hence,
 \begin{gather*} \textstyle{
 \frac{\partial [b_1, b_2]^n}{\partial b_1} = (\sum_{j=0}^{n-1}[b_1, b_2]^j)(1  - b_1b_2b_1^{-1}) \ , } \\ {} \qquad
 \textstyle{ \frac{\partial [b_1, b_2]^n}{\partial b_2} = (\sum_{j=0}^{n-1}[b_1, b_2]^j)(b_1  - b_1b_2b_1^{-1}b_2^{-1}) \ . }
\end{gather*}
 Therefore, taking the image of the equality \eqref{FD} in $\mathbb Z[B(2,n)]$, we obtain
\begin{gather*} \textstyle{
 0 = [a_1, a_2]^n -1  =   (\sum_{j=0}^{n-1}[a_1, a_2]^j)(1  - a_1a_2a_1^{-1}) (a_1-1)+} \\  \textstyle{  +   (\sum_{j=0}^{n-1}[a_1, a_2]^j) (a_1  - a_1a_2a_1^{-1}a_2^{-1}) (a_2-1)
 } \ .
\end{gather*}
Now multiplication on the right by $\sum_{i=0}^{n-1} a_2^i$ yields
\begin{gather*} \textstyle{
   (\sum_{j=0}^{n-1}[a_1, a_2]^j) (1  - a_1a_2a_1^{-1}) (a_1-1) (\sum_{i=0}^{n-1} a_2^i) =  0
 } \
\end{gather*}
and this is what we have in Theorem~1.

Analogously, let a group $G = \langle a_1,a_2 \rangle $ be generated by $a_1, a_2$, $a_2^n=1$ in $G$,   a word $w(b_1,b_2)\in F(b_1,b_2)$ have the
property that $w(a_1,a_2)=1$ in $G$, and  $\theta~:~\mathbb Z[F(b_1,b_2)]\to \mathbb Z[G]$, where $\theta (b_i) = a_i$, $i=1,2$,  denote the natural epimorphism. As above, we can obtain
$$\textstyle{
\theta( \frac{\partial w(b_1,b_2)}{\partial
b_1}(b_1-1)(\sum_{j=0}^{n-1}b_2^j)  )=  \theta( \frac{\partial w(b_1,b_2)}{\partial
b_1})(a_1-1)(\sum_{j=0}^{n-1}a_2^j) =    0 \ . }
$$
This equation can be used  for  constructing other potentially nontrivial pairs of zero-divisors in   $\mathbb Z[G]$  (which, however, does not work in case when $G$ is a free product of the form $\langle a_1 \rangle * \langle a_2\rangle$).
\bigskip

{\em Proof of Theorem~2.}
Suppose $G$ is a group and $G$ contains a subgroup $H$ isomorphic either to a finite noncyclic group or to the free product $C_{q}* C_{r}$ of cyclic groups  $C_{q}, C_{r}$, where $1< \min(q, r) < \infty$.

First assume $H$ is a finite noncyclic group. Then there are $h_1, h_2 \in H$ such that the subgroup $\langle h_1, h_2 \rangle$, generated by  $h_1, h_2$, is not cyclic.
Since $2-h_1 - h_2$ is a left (right) zero-divisor in $\mathbb Z[H]$,
$2-h_1 - h_2$ is also a left (right, resp.) zero-divisor in $\mathbb Z[G]$.
If $2-h_1 - h_2$ is a trivial left (right, resp.) zero-divisor  in $\mathbb Z[G]$, then it follows from Lemma~\ref{lem1} that elements  $1, h_1, h_2$ belong to the same coset  $gH_0$  ($H_0g$, resp.), where $H_0= \langle h_0 \rangle$ is cyclic. But then $g \in H_0$ and $ h_1, h_2 \in H_0$, whence
the subgroup  $\langle h_1, h_2 \rangle$ is  cyclic. This contradiction completes the proof in the case when $H$ is finite noncyclic.

Suppose $C_{q}* C_{r}$ is a subgroup of $G$, $1< \min(q, r) < \infty$.
We may assume that $q $ is finite. Denote $C_{q} = \langle a  \rangle_q$ and    $C_{r} = \langle b  \rangle_r$. Note the subgroup $\langle a, babab  \rangle$ of  $C_{q}* C_{r}$ is isomorphic to the free product $C_{q}* C_{\infty}$  unless $q=r= 2$. Therefore, we may assume that $G$ contains a subgroup
isomorphic to $C_{q'}* C_{r'}$, where
$q'=q >1$ is finite and either  $r'= \infty$ or $q'=r'= 2$.  Now Theorem~\ref{th2} follows from Lemma~\ref{lem3}. \qed


\begin{thebibliography}{[75]}

\bibitem[1]{C1}
P. M. Cohn,
{\em   On the free product of associative rings. II. The case of (skew) fields},  Math. Z. {\bf 73}(1960), 433--456.

\bibitem[2]{C2}
P. M. Cohn,
{\em   On the free product of associative rings. III},
J. Algebra {\bf 8}(1968), 376--383.

\bibitem[3]{DS}
Michael A. Dokuchaev and Maria Lucia Sobral Singer,
{\em  Units in group rings of free products
of prime cyclic groups}, Can. J. Math. {\bf  50}(1998), 312--322.


\bibitem[4]{Ger}
V. N. Gerasimov,
{\em   The group of units of a free product of rings},
 {Mat. Sbornik} {\bf 134}(1989), 42--65.

\bibitem[5]{Iv92}   S. V. Ivanov, {\em Strictly verbal
products of groups and a problem of Mal'tcev on operations on
groups},  { Trans. Moscow Math. Soc.} {\bf 54}(1992),
217--249.

\bibitem[6]{Iv94}  S. V. Ivanov,  {\em The free Burnside groups of sufficiently
large exponents}, {Internat. J. Algebra Comp.}  {\bf
4}(1994), 1--308.

\bibitem[7]{Iv99} S. V.  Ivanov,   {\em An asphericity conjecture and Kaplansky
problem on zero divisors}, { J. Algebra}  {\bf 216}(1999),
13--19.

\bibitem[8]{Iv03}  S. V. Ivanov, {\em  On subgroups of free
Burnside groups of large odd exponent}, { Illinois J. Math.}
{\bf 47}(2003), 299--304.


\bibitem[9]{Iv05} S. V. Ivanov,  {\em Embedding free Burnside groups in
finitely presented groups},  {Geom. Dedicata} {\bf
111}(2005), 87--105.


\bibitem[10]{KNT}  {\em Kourovka Notebook: Unsolved problems in group theory}, Novosibirsk, 11th Ed., 1990.

\bibitem[11]{Kur} A. G. Kurosh, {\em The theory of groups},
Chelsea, 1956.

\bibitem[12]{Ol82}
A. Yu. Ol'shanskii, {\em On the Novikov-Adian theorem}, {Mat.
Sbornik} {\bf 118}(1982), 203--235.


\bibitem[13]{OB}
A. Yu. Ol'shanskii,  {\em Geometry of defining relations in
groups}, Nauka, Moscow, 1989; English translation: {\em Math.
and Its Applications, Soviet series},  vol. 70, Kluwer Acad.
Publ., 1991.

\end{thebibliography}
\end{document}